%% file: GreenWedgeSubmit.tex
\documentclass[12pt]{article}

\usepackage{amsmath,bbm,amssymb,amsthm,epsfig}
\usepackage{graphicx}
\include{definitions}
\theoremstyle{definition}
\newtheorem{thm}{Theorem}
\newtheorem{lem}{Lemma}

\newtheorem{prop}{Proposition}
\include{defns}

\begin{document}


\title{Rates of Convergence for the Planar Discrete Green's Function in Pacman Domains}
\author{Christian Bene\v{s}
}
\date{}
\maketitle

\begin{abstract}
We obtain upper bounds for the rates of convergence for the simple random walk Green's function in the domains $D_\alpha = D_{\alpha}(n)=\{re^{i\theta}\in \C:0 <\theta<2\pi-\alpha, 0<r<2n\}-z_0,$
where $z_0\in\Z^2$ is a point closest to $ne^{i(\pi-\alpha/2)}$. The rate depends on the angle of the wedge and is what was suggested by the sharpest available results in the extreme cases $\alpha =0$ and $\alpha=\pi$. Our proof uses the KMT coupling between random walk and Brownian motion.
\end{abstract}

\section{Introduction and Statement of Main Result}

\subsection{Motivation}

The rate of convergence of the Green's function for different kinds of planar, discrete-time random walks to the continuous Green's function has been studied in a number of papers in general classes of domains (see for instance \cite{kl}, \cite{bjk}, and \cite{jiangkennedy}), as well as in some specific domains, in the case of the simple random walk Green's function, such as the disk (see \cite{greenbook}) and the half- and quarter-plane (see \cite{spitzer}). The currently available results suggest that in the case of smooth domains, the rate of convergence should be of the same order as that of the lattice spacing, whereas in arbitrary domains, it should be of order square root of the lattice spacing. The present paper suggests an answer to the question of how the rate of convergence depends on domain regularity by examining a family of domains with one singular boundary point.

The question of rate of convergence of discrete Green's functions is intimately related to the question of the rate of convergence of discrete harmonic measure (see \cite{SLEbook} for a discussion of the relation between the Green's function, harmonic measure, and the Poisson kernel, the Radon-Nikodym derivative of harmonic measure with respect to Lebesgue measure). It is shown in \cite{jiangkennedy} that the rate of convergence  in smooth domains of harmonic measure for the discrete-time, continuous-space random walk considered in that paper is the same as the rate of convergence of the corresponding Green's function, that is, roughly the inverse of the step size. However, no equivalent result is currently available for simple random walk. The work done in the present paper is intended to be a first step towards obtaining rates of convergence for simple random walk harmonic measure.
Note that discrete harmonic measure is known to converge in the domains we are examining in this paper (see, e.g., \cite{lawlerlimic}, \cite{chelkaksmirnov}, and \cite{biskuplouidor}, which all discuss classes of domains which contain the domains considered in the present paper).

The importance of harmonic measure itself is manifold, but particularly obvious in the context of the Dirichlet problem: The Dirichlet problem in a domain $D\subset\C$ consists of finding a function $f:\bar{D}\to\R$, harmonic in $D$, with prescribed boundary values $h$, that is, of finding a function $f:\bar{D}\to\R$ such that 
\begin{equation}\label{dirichlet}
\left\{
\begin{array}{ll}
  \Delta f(z)  
  =0, & z\in D\\
  f(z) = h(z), & z\in\partial D.
\end{array}
\right.
\end{equation}
Whether the Dirichlet problem actually has a solution depends both on $D$ and $h$. It is known (see \cite{garnett_marshall} for an analytic point of view and \cite{SLEbook} for a probabilistic approach) that if $D$ is a regular domain (this roughly means that all points in $\partial D$ are part of a piece of a curve $\subset \partial D$ containing more than one point; in particular all simply connected domains are regular) and $h$ is continuous and bounded, then there exists a unique bounded, continuous solution to the Dirichlet problem. In that case, one can write the solution of \eqref{dirichlet} as
$$f(z)=\int_{\partial D} h(w)\om(z,|dw|; D),$$
where $\om=\om(z, \cdot; D)$ is harmonic measure from $z$ on $\partial D$.


\subsection{Definitions and Important Properties}

For a domain $D\subsetneq \Z^2$, if $S$ is simple random walk started at $z$ and
\begin{equation}\label{T}
T_D=\min\{k\geq 0:S_k\not\in D\},
\end{equation} 
is the first time $S$ leaves $D$, the discrete Green's function in $D$ is, for $z,w\in\Z^2$,
$$G_D(z,w)=E^z\left[\sum_{k\geq 0}\1\{S_k=w; k <T_D\}\right],$$
the expected number of visits to $w$ before leaving $D$ by $S$ started at $z$. We will write $G_D(w)=G_D(0,w)$.

A representation of $G_D$ which will be particularly useful for us is the following (see \cite{greenbook}): For $z,w\in\Z^2$,
\begin{equation}\label{discretegreenrep}
G_D(z,w)=E^z[a(S(T_D)-w)]-a(z-w),
\end{equation}
where for $x\in\Z^2$, 
$$a(x)=\sum_{j\geq 0}(P^0(S(j)=0)-P^x(S(j)=0))$$
is the potential kernel for simple random walk. As $|x|\to\infty$, $a$ has the representation
\begin{equation}\label{ax}
a(x)=\frac{2}{\pi}\log|x|+k_0+\bigo{|x|^{-2}},
\end{equation}
where $k_0=\frac{2\gamma+3\ln 2}{\pi}$ and $\gamma$ is Euler's constant. See \cite{FukU1} for more details.

One can define a continuous analogue of the discrete Green's function.
If $D\subsetneq \C$ is a domain such that for any $z\in D$, if $B$ is standard Brownian motion started at $z$, 
\begin{equation}\label{tau}
\tau_D=\inf\{t\geq 0:B_t\not\in D\}
\end{equation}
satisfies $\tau_D<\infty, \text{a.s.}$, we can define
$$p_D(t,z,w)=\lim_{\eps\to 0}\frac{1}{\pi\eps^2}P^z(|B_t-w|\leq \eps, t\leq \tau_D),$$
the transition density for $B$ from $z$ to $w$ before exiting $D$. The Green's function in $D$ is then, for $z\neq w$, 
$$g_D(z,w)=\pi\int_0^{\infty} p_D(t,z,w)\,dt.$$
It is the unique harmonic function on $D\setminus \{z\}$ satisfying $\lim_{w\to w_0}g_D(z,w)=0$ for every regular (see \cite{SLEbook} for a precise definition) boundary point $w_0\in \bd D$ and
$$g_D(z,w) + \log|z-w| = \bigo{1}, \qquad \text{ as } |z-w|\to 0$$
(see below for the definition of $\bigo{\cdot}$). We will write $g_D(w)=g_D(0,w)$.

An analogue of \eqref{discretegreenrep} holds for $g_D$: For $z,w\in D$,
\begin{equation}\label{contgreenrep}
g_D(z,w)=E^z[\log |B(\tau_D)-w|]-\log |z-w|.
\end{equation}
Note that \eqref{contgreenrep} implies that $g_D$ is unchanged under re-parametrizations of $B$.

The fundamental property of conformal invariance of planar Brownian motion carries over to the Green's function: If $\psi:D\to D'$ is a conformal transformation, then 
$$g_D(z,w)=g_{\psi(D)}(\psi(z),\psi(w)).$$




Throughout this paper, when we write $f(z)=\bigo{g(z)}$, we mean that there exists a constant $c$ such that for all $z$ in a set which will depend on the context (usually, it will be for $z$ large enough or for $z$ small enough), $|f(z)|\leq c|g(z)|$. We will use this notation for real-valued functions only and make it explicit when we need it for complex-valued functions. We will also use the notation $f(z)\lesssim g(z)$ to mean the same thing and $f(z)\gtrsim g(z)$ to mean $g(z)=\bigo{f(z)}$. $f(z)=o(g(z))$ will mean $\lim_{|z|\to\infty}f(z)/g(z)=0$.

The proof of Theorem A.1 in \cite{bjk} suggests that the worst-case scenario in arbitrary domains arises when the boundary contains a slit and that the rate of convergence of the Green's function is fastest when the domain is smooth. 
It is therefore natural to consider the domains (see Figure \ref{Fig1})
\begin{equation}\label{da}
D_\alpha = D_{\alpha}(n)=\{re^{i\theta}\in \C:0 <\theta<2\pi-\alpha, 0<r<2n\}-z_0,
\end{equation}
where $0\leq \alpha\leq \pi$ and $z_0\in\Z^2$ is a point closest to $ne^{i(\pi-\alpha/2)}$. Note that these domains have inner radius within one unit of $n$.

\subsection{Rate of Convergence of Discrete Green's Function}

This paper's main result is an upper bound for the rate of convergence of $G_{D_{\alpha}}(w)$ for $w\in D_{\alpha}$:

\begin{thm}\label{green}
If $w\in D_{\alpha}$,
$$\left|G_{D_{\alpha}}(w)-\frac{2}{\pi}g_{D_{\alpha}}(w)\right|=\bigo{\left(\frac{\log^2 n}{n}\right)^{c_{\alpha}}+|w|^{-2}},$$
where 
\begin{equation}\label{ca}
c_{\alpha}=\frac{1}{2}+\frac{\alpha}{4\pi-2\alpha}=\frac{\pi}{2\pi-\alpha}.
\end{equation}
In particular, for $|w|\geq (n/\log^2 n)^{c_\alpha/2}$,
$$\left|G_{D_{\alpha}}(w)-\frac{2}{\pi}g_{D_{\alpha}}(w)\right|=\bigo{\left(\frac{\log^2 n}{n}\right)^{c_{\alpha}}}.$$
\end{thm}

\section{Proof of Theorem \ref{green}}

\subsection{Exit Probabilities for $S$ and $B$ in $D_\alpha$}

Central to our proof is a coupling between simple random walk and Brownian motion, called the KMT approximation (see \cite{kmt2} and \cite{benesnotes} for a simple argument justifying the extension of the result from dimension one to two) of which we state a consequence in Lemma \ref{kmt} below. In this coupling of planar simple random walk and standard planar Brownian motion, it is random walk at time $2k$ and Brownian motion at time $k$ which are close to each other with high probability. For notational convenience, for the rest of this paper, we let $\tilde{B}$ be standard planar Brownian motion and define for all $t\geq 0$,
\begin{equation}\label{reparam}
B(t) = \tilde{B}(t/2).
\end{equation}
Note that changing the speed of Brownian motion linearly doesn't affect its path properties and, as mentioned before, leaves the corresponding Green's function unchanged.
We also consider $S$ to be interpolated linearly between integer times, that is, for all $t\in \R_+$, 
\begin{equation}\label{interpol}
S_t=S_{\lfloor t\rfloor}+(t-\lfloor t\rfloor)(S_{\lceil t\rceil}-S_{\lfloor t\rfloor}).
\end{equation}
With these definitions, we have
\begin{lem}\label{kmt}
There exist $c_0$ and a probability space containing a planar Brownian motion $B$ as in \eqref{reparam} and a two-dimensional simple random walk $S$ as in \eqref{interpol} such that for all $n$, if $D$ is any set with outer radius at most $3n$, that is, such that $\sup\{|z|:z\in D\}\leq 3n$, then for any $z\in D$,
$$P^z\left(\sup_{0\leq t\leq T_D\vee\tau_D}|S_t-B_t| > c_0\log n\right)=\bigo{n^{-10}},$$
where $P^z$ is the measure associated with $B$ and $S$ both started at $z$, $T_D$ is as in \eqref{T}, and $\tau_D$ is as in \eqref{tau}, but for the reparametrized Brownian motion.
 
\end{lem}




The following basic estimates will be helpful in obtaining our key estimates below. The first and third follow from estimates for standard Brownian motion that can be found, e.g., in \cite{benesnotes} and the second is obvious.

\begin{lem}\label{toofar} If $B$ is planar Brownian motion as in \eqref{reparam} and $S$ is two-dimensional simple random walk, then there exist constants $C$ and $K$ such that 
\begin{enumerate}
\item[(a)] $P(\sup_{0\leq t \leq 1}|B(t)| \geq r) \leq Ce^{-r^2}.$
\item[(b)] $P(\sup_{0\leq t \leq 1}|S(t)| \geq r) = 0 \text{ if } r>1.$
\item[(c)] $P(\sup_{0\leq t \leq n}|B(t)| \leq r^{-1}n^{1/2}) \leq \exp\{-K r^2\}.$
\end{enumerate}
\end{lem}

At the center of our argument are the following lemmas which estimate the probability of $B$ or $S$ leaving $D_{\alpha}$ in some small subset of the boundary. The first follows from conformal invariance of planar Brownian motion and the second is obtained from the first using the KMT coupling of Lemma \ref{kmt}.

Recall the definition of $D_\alpha$ in \eqref{da} and of $z_0$ in the following line. Let
$$N=\lceil 2n/\log^2 n \rceil$$
and define for $1\leq k\leq N$,
$$
I_k=\{z\in\bd D_\alpha: (k-1)\log^2 n\leq |z+z_0| < k\log^2 n\}.
$$
Note that since by $\log n$ we mean the natural logarithm, $n/\log^2 n$ cannot be an integer, so that $I_N$ contains the circular part of  $\bd D_{\alpha}$. Recall the definition of $c_\alpha$ in \eqref{ca}.

\begin{lem}\label{beurling1} 
Fix $\alpha\in [0,\pi]$ and $a>0$. Then for all $n$ large enough, 
all $x\in D_{\alpha}$ satisfying $d(x,\bd \Da)\leq a\log n$ and $d(x,\bd \Da)=d(x,I_{k_0})$, and all $1\leq k\leq N$ for which $|k_0-k|\geq 2$, 
$$P^x(B_{\tau_{D_{\alpha}}}\in I_k)\lesssim \frac{(k_0k)^{\ca-1}}{(k^{\ca}-k_0^{\ca})^2\log^{c_\alpha} n}.$$
\end{lem}

 \begin{figure}[htb!]
\centering%
\includegraphics[scale=0.84]{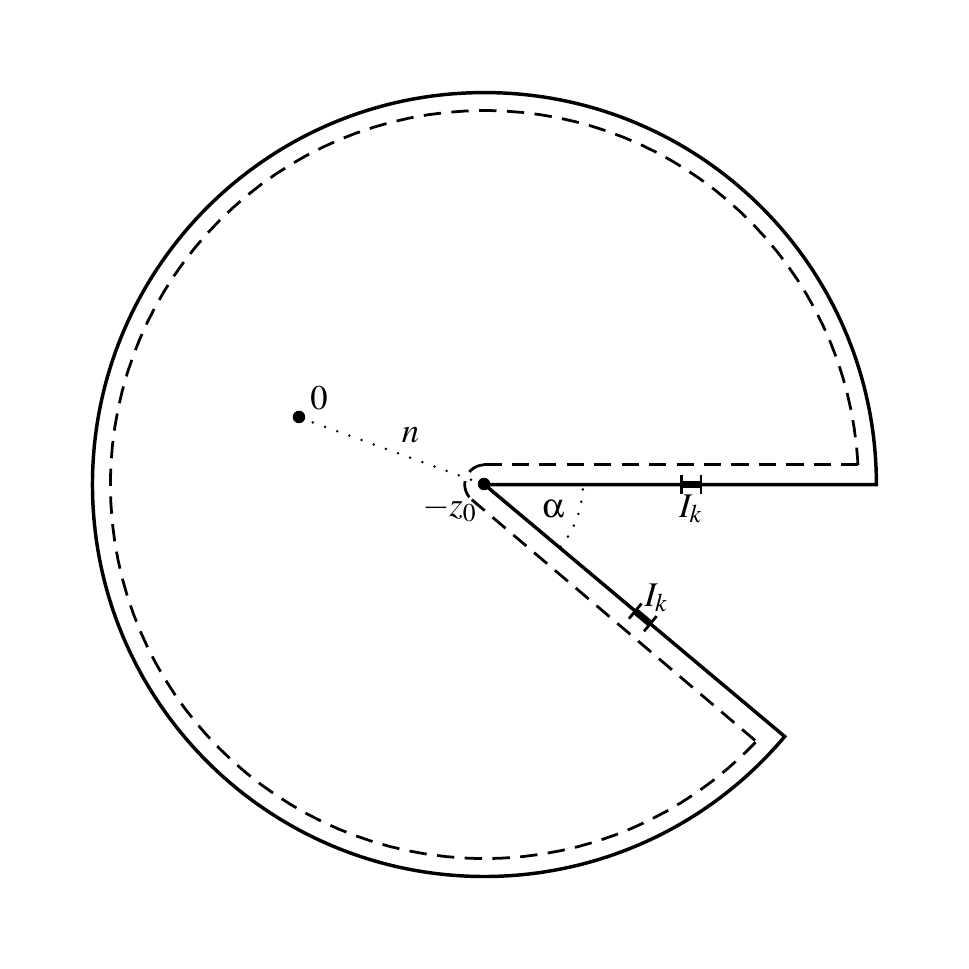}
\caption{
$\Da$, together with $I_k$ for some arbitrary $1<k<N$. The dashed line is the set of points in $\Da$ at distance $2c_0\log n$ from $\bd \Da$
corresponding to the hitting time $\eta$ in the proof of Theorem \ref{green}.}
\label{Fig1}
\end{figure}

\begin{proof} We will assume without loss of generality that 
$x$
satisfies $\arg (x+z_0)\in [0,\pi-\alpha/2]$, in other words, is in the ``top half" of $\bd D_{\alpha}$.

We will consider first the case where $k<N, k_0<N$. Note that the map $f=f_\alpha$ defined by 
$$f_{\alpha}(z)=\left(\frac{z+z_0}{2n}\right)^{c_{\alpha}}$$
sends the domain $D_{\alpha}$ to the unit upper half-disk 
\begin{equation}\label{d+}
\D^+=\{z\in\C:\im(z)>0, |z|<1\}
\end{equation} 
and satisfies 
$$f(0)=\left(\frac{1}{2}\right)^{c_\alpha}i\left(1+\bigo{n^{-1}}\right).$$ 
It is also easy to verify that if for $2\leq k_0\leq N-1$, $d(x,\bd\Da)=d(x,x_0)$ with $x_0\in  I_{k_0}$, then
$$f(x)=\frac{x_0^{c_{\alpha}}}{(2n)^{c_{\alpha}}}\left(1+(1+i)\bigo{\frac{\log n}{x_0}}\right),$$
 and if $d(x,\bd D_{\alpha})=d(x,I_1)$, then
$$0\leq \re(f(x))\leq \left(\frac{\log^2 n}{n}\right)^{c_{\alpha}}, \quad 0\leq \im(f(x)) \lesssim \left(\frac{\log n}{n}\right)^{c_{\alpha}}.$$
Moreover,
$$f(I_k)=\left(-\left(\frac{k}{2n}\right)^{c_{\alpha}}\log^{2c_{\alpha}}n, -\left(\frac{k-1}{2n}\right)^{c_{\alpha}}\log^{2c_{\alpha}}n\right] \cup \left[\left(\frac{k-1}{2n}\right)^{c_{\alpha}}\log^{2c_{\alpha}}n, \left(\frac{k}{2n}\right)^{c_{\alpha}}\log^{2c_{\alpha}}n\right).$$
Then, with $\D+$ as in \eqref{d+}, for all $x$ with $d(x,\bd\Da)\leq a\log ^2 n, d(x,\bd\Da)=d(x,I_{k_0})$ for some $1\leq k_0<N$, 
\begin{eqnarray}\label{gen}\notag
 P^x\left(B(\tau_{D_\alpha})\in I_k\right) &=& P^{f(x)}\left(B(\tau_{\D^+})\in f(I_k)\right) \leq  P^{f(x)}\left(B(\tau_{\h})\in f(I_k)\right)\\ \notag
& \lesssim & P^{i}\left(B(\tau_{\h})\in \left[\log^{c_\alpha} n k_0^{1-c_{\alpha}}|(k-1)^{c_{\alpha}}-k_0^{c_{\alpha}}|,\log^{c_\alpha} n k_0^{1-c_{\alpha}}|k^{c_{\alpha}}-k_0^{c_{\alpha}}|\right)\right)\\ 
& \lesssim & \frac{(k_0k)^{\ca-1}}{(k^{\ca}-k_0^{\ca})^2\log^{c_\alpha} n}.
\end{eqnarray}
where the equality follows from conformal invariance of planar Brownian motion, the first inequality follows from translation and scaling invariance of Brownian motion and the last inequality follows from the fact that for Brownian motion started at $i$, $B(\tau_{\h})$ has the Cauchy distribution, the equality $\arctan(x)+\arctan(1/x)=\frac{\pi}{2}$, and the Taylor expansion of $\arctan$ at the origin. In fact, the computation yields the better bound of $\frac{(k_0k)^{\ca-1}}{(k^{\ca}-k_0^{\ca})^2\log n}$ for all $k_0\geq 2$, but not for $k_0=1$.

Consider now the case where $k_0=N$ and let $t_{x_0}$ be the tangent line to $\bd D_{\alpha}$ at $x_0$, a closest point to $x$ in $\bd\Da$ and $\h_{x_0}$ the half-plane with boundary $t_{x_0}$ containing the origin. Then 
\begin{equation}\label{d}
d:=\inf\{|w-x_0|:w\in I_k\}\gtrsim (N-k)\log^2 n. 
\end{equation}
The strong Markov property applied at time $\tau_{D\left(x_0,\frac{d}{2}\right)}$ implies that
\begin{equation}\label{12-3}
P^x(B(\tau_{D_\alpha})\in I_k)  \leq  P^x\left(\tau_{D\left(x_0,\frac{d}{2}\right)\cap D_{\alpha}} < \tau_{D_{\alpha}}\right)\sup_{z\in D\left(x_0,\frac{d}{2}\right)}P^z(B(\tau_{D_\alpha})\in I_k).
\end{equation}
If $x_k^{(1)}, x_k^{(2)}$ are the midpoints of the two segments forming $I_k$, we can use the fact that since $x_0\in I_N, \max\{d,x_k^{(i)}\}\asymp n$ for $i=1,2,$ to easily verify that there is a constant $C$ such that for $i=1, 2$, 
$$f\left(D\left(x_k^{(i)},\frac{d}{2}\right)\cap D_\alpha\right) \supset D\left(f(x_k^{(i)}),C\frac{d}{n}\right)\cap\D_+.$$
Therefore, since the width of each segment of $f(I_k)$ is 
$$f(k\log^2 n-z_0)-f((k-1)\log^2 n-z_0) \lesssim \left(\frac{k\log^2 n}{2n}\right)^{c_\alpha}k^{-1},$$ we have
\begin{eqnarray}\label{second} \notag
\sup_{z\in D\left(x_0,\frac{d}{2}\right)}P^z(B(\tau_{D_\alpha})\in I_k) & \leq & \sup_{z\in D\left(x_k^{(1)},\frac{d}{2}\right)^c\cap D\left(x_k^{(2)},\frac{d}{2}\right)^c\cap \Da}P^z(B(\tau_{D_\alpha})\in I_k)\\ \notag
&\lesssim & \sup_{z\in D\left(f(x_k^{(1)}),\frac{d}{n}\right)^c\cap D\left(f(x_k^{(2)}),\frac{d}{n}\right)^c\cap\D_+}P^z(B(\tau_{\h})\in f(I_k))\\
&\lesssim & \frac{\left(\frac{k\log^2 n}{n}\right)^{c_\alpha}k^{-1}}{dn^{-1}}\lesssim \left(\frac{k}{N}\right)^{c_\alpha-1}\frac{\log^2 n}{d},
\end{eqnarray}
where the penultimate 
inequality follows from the fact that the exit distribution of the half-plane has the Cauchy distribution.

Moreover, 
\begin{equation}\label{first}
P^x\left(\tau_{D\left(x_0,\frac{d}{2}\right)\cap D_{\alpha}} < \tau_{D_{\alpha}}\right)  \leq  P^x\left(\tau_{D\left(x_0,\frac{d}{2}\right)\cap \h_{x_0}} < \tau_{D_{\alpha}}\right) \lesssim \frac{\log n}{d},
\end{equation}
where the second inequality is essentially the Gambler's ruin estimate but can be shown rigorously, via conformal invariance, using the fact that the map $-(z+z^{-1})$ is a conformal transformation of the upper unit half-disk into the upper half-plane and, again, the fact that the exit distribution of the half-plane has the Cauchy distribution.

Plugging \eqref{second} and \eqref{first} into \eqref{12-3} and using \eqref{d} now implies that
\begin{equation}\label{extreme}
P^x(B(\tau_{D_\alpha})\in I_k)  
\lesssim \left(\frac{k}{N}\right)^{c_\alpha-1}\frac{1}{(N-k)^2\log n}.
\end{equation}

The lemma now follows from the fact that the bound in \eqref{extreme} is of order at most that in \eqref{gen} when $k_0=N$.

\end{proof}

\begin{lem}\label{beurling2} Fix $\alpha\in [0,\pi]$ and $a>0$. Then for all $n$ large enough, 
all $x\in D_{\alpha}\cap\Z^2$ satisfying $d(x,\bd \Da)\leq a\log n$ and $d(x,\bd \Da)=d(x,I_{k_0})$, and all $1\leq k\leq N$ for which $|k_0-k|\geq 2$, 
$$P^x(S_{T_{D_{\alpha}}}\in I_k)\lesssim \frac{(k_0k)^{\ca-1}}{(k^{\ca}-k_0^{\ca})^2\log^{c_\alpha} n}.$$
\end{lem}

\begin{proof}
We use the KMT coupling of Lemma \ref{kmt} to derive this estimate from the analogous estimate for Brownian motion in Lemma \ref{beurling1}.

Assume first that $\alpha\neq 0$ and define 
$$z_0'=z_0-2c_0\log n(\cot(\alpha/2)-i).$$
Note that $-z_0'\not\in\Da$ is the point of intersection of two lines that are parallel to the segments of $\partial \Da$ and at distance $2c_0\log n$ of those segments. We then define
$$\Da' = \{z\in\C:|z+z_0|\leq 2n+2c_0\log n\}\setminus \{x\in\C:\arg(z+z_0')\in (-\alpha,0)\}$$
and, for $1\leq k\leq N$, 
$$I'_k=\{z\in\bd \Dap:d(z,\bd\Da)=d(z,I_k)\}.$$ 
We assume for the rest of the proof that $n$ is large enough so that $I'_k\neq\emptyset$ for all $1\leq k\leq N$.


We now couple $B$ and $S$ as in Lemma \ref{kmt} and define 
\begin{equation}\label{k}
\mathcal{K}=\left\{\sup_{0\leq t\leq T_{D_{\alpha}}\vee\tau_{D_{\alpha}}}|S_{t}-B_t| \leq  c_0\log n\right\}.
\end{equation}
Then, by Lemma \ref{kmt}, 
\begin{equation}\label{kc}
P(\K^c)=\bigo{n^{-10}}.
\end{equation}
For $1\leq k\leq N$, let
$$\mathcal{R}'_k = \{z\in \Da':d(z,I_{k}')\leq 10c_0\log n\}$$
and note that $I_k\subset \mathcal{R}'_k$ and for any $z\in I_k, d(z,\partial \mathcal{R}'_k)\geq 2c_0\log n$. By Lemma \ref{toofar} (b),
$$\{S_{T_{\Da}}\in I_k, T_{\Da}\in [j-1,j), \K\}\subseteq \{B_{j-1}\in \mathcal{R}'_k, \K \}
.$$
Then, by \eqref{kc}, since on $\K, \tau_{\Da'}\geq T_{\Da}$,
\begin{eqnarray}\label{main3}\notag
P^x(S_{T_{D_{\alpha}}}\in I_k) & \leq &  \sum_{j\geq 1}P^x(S_{T_{D_{\alpha}}}\in I_k, T_{D_{\alpha}}\in [j-1,j), \K)+\bigo{n^{-10}}\\ \notag
&\leq & \sum_{j\geq 1} P^x(S_{T_{D_{\alpha}}}\in I_k, T_{D_{\alpha}}\in [j-1,j), \tau_{D_{\alpha}'}\geq j-1, B_{j-1}\in\mathcal{R}'_k, \K)+\bigo{n^{-10}}
\\ \notag
&\lesssim & \sum_{j\geq 1}P^x(T_{D_{\alpha}}\in [j-1,j), \tau_{D_{\alpha}'\setminus \mathcal{R}'_k}\leq \tau_{D_{\alpha}'})+\bigo{n^{-10}}\\  
&= & P^x(\tau_{\Dap\setminus \mathcal{R}'_k}\leq \tau_{D_{\alpha}'})
+\bigo{n^{-10}}
\end{eqnarray}
 
 With the convention $I'_0=I'_{N+1}=\emptyset$, we let, for $1\leq k\leq N$,
$$I_{k,+}'=\cup_{i=k-1}^{k+1}I_i'.
$$ 

 We now claim that 
\begin{equation}\label{comparable}
P^x(\tau_{\Dap\setminus \mathcal{R}_k}\leq \tau_{D_{\alpha}'})\lesssim P^x(B_{\tau_{D_{\alpha}'}}\in I_{k,+}').
\end{equation}
Indeed, by the strong Markov property for Brownian motion,
\begin{eqnarray*}
P^x(\tau_{\Dap\setminus \mathcal{R}'_k}\leq \tau_{D_{\alpha}'})&\leq & P^x(B_{\tau_{D_{\alpha}'}}\in I_{k,+}')+P^x(\tau_{\Dap\setminus \mathcal{R}'_k}\leq \tau_{D_{\alpha}'},B_{\tau_{D_{\alpha}'}}\not\in I_{k,+}')\\
&=&P^x(B_{\tau_{D_{\alpha}'}}\in I_{k,+}')+P^x(\tau_{\Dap\setminus \mathcal{R}'_k}\leq \tau_{D_{\alpha}'})\bigo{\log^{-1} n}.
\end{eqnarray*}
Equations \eqref{main3} and \eqref{comparable} now imply
$$P^x(S_{T_{D_{\alpha}}}\in I_k)  \lesssim P^x(B_{\tau_{D_{\alpha}'}}\in I_{k,+}') +\bigo{n^{-10}}$$
and the lemma now follows from an slight modification of Lemma \ref{beurling1} to $D_{\alpha}'$ (note that $D_{\alpha}'$ is not just a rescaled version of $D_{\alpha}$, so Lemma \ref{beurling1} cannot be applied directly but the argument of the proof of that lemma yields the same bound for $P^x(B_{\tau_{D_{\alpha}'}}\in I_k')$ as in Lemma \ref{beurling1}).

If $\alpha = 0$, we need to use a slightly different argument from the one we just used. For a set $D\subset \C$, we let 
$$\sigma_D=\tau_{D^c}=\inf\{t\geq 0: B_t\in D\}$$ 
be the first hitting time of $D$ by $B$.
We will write $D_0$ for the set $D_\alpha$ with $\alpha=0$ and for $z\in\C, r\in\R_+$, we let $C(z,r)$ be the circle of radius $r$, centered at $z$. With $a$ as in the statement of the lemma, we let 
$$
b=\max\{ac_0,2c_0\}
$$
and define
$$\mathcal{S}=\{z\in \C:d(z,\bd D_0)\leq b\log n\}\setminus\{z\in\C:|\im(z+z_0)|\leq b\log n,|\re(z+z_0)|\leq b\log n\},$$
$$\S_{\text{top}}=\{z\in \bd\S\cap D_0:\im(z+z_0)=b\log n\} \cup C(-z_0,2n-b\log n),$$ 
$$\S_{\text{bot}}=\{z\in \bd\S\cap D_0:\im(z+z_0)=-b\log n\} \cup C(-z_0,2n+b\log n),$$
$$\S_{\text{end}}=\{z\in\C:|\im(z+z_0)|\leq b\log n,|\re(z+z_0)|\leq b\log n\},$$
$$\L_1= \{z\in D_0\cup\S:\im(z+z_0)=0, \re(z+z_0)\leq b\log n\},$$
For $1\leq k\leq N$, let
$$\mathcal{S}_k=\{z\in \S:d(z,\bd D_0)=d(z,I_k)\}$$
and for $2\leq k\leq N-1$, define
$$\mathcal{S}_k^+=\S_{k-1}\cup\S_{k}\cup\S_{k+1}.$$
 \begin{figure}[htb!]
\centering%
\includegraphics[scale=0.84]{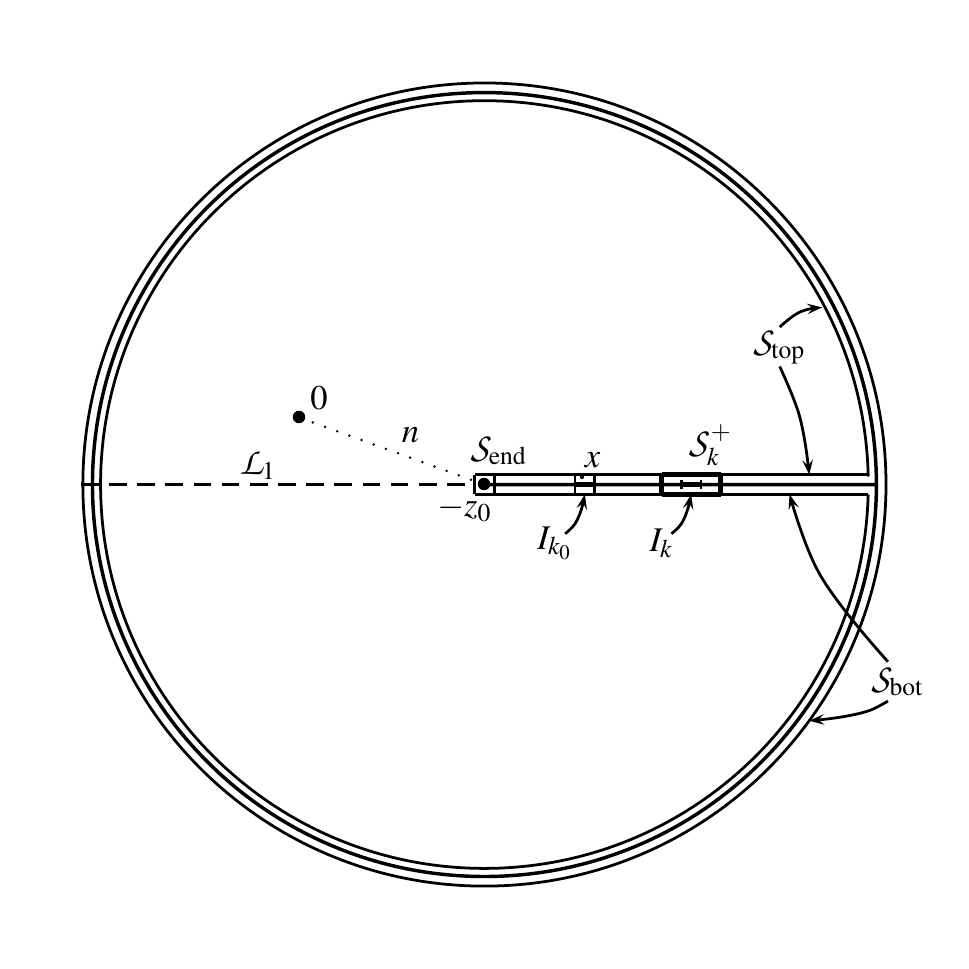}
\caption{
The protagonists of the proof of Lemma \ref{beurling2} in the case $\alpha=0$.}
\label{Fig2}
\end{figure}
Note that $b$ is defined in such a way that if $x$ is as in the statement of the lemma, then $x\in\S\cup\mathcal{D}$. We assume for the rest of the proof, without loss of generality, that $\im(x+z_0)\geq 0$. The main idea of this proof is to start $B$ and $S$ from $x$ and couple them as in \eqref{kmt} and note that in order for $S$ to leave $D_0$ at $I_k, 2\leq k\leq N-1$, $B$ can't 
reach $\S_{\text{bot}}$ before entering $\mathcal{S}_k^+$ or hitting $\L_1
$, since in that case, $S$ would necessarily either have hit $\R_+$ at a point outside of $I_K$ or have hit the circle $C(-z_0,2n)$ and would therefore have left $D_0$ without hitting $I_k$. 

We let
$$\eta\!=\! \inf\{t\geq 0\!:\! \exists \, s\leq t \text{ s.t. } \!\! B_s\in \S_{\text{top}}, B_t\!\in \S_{\text{bot}}, B[s,t]\subseteq \S\! \text{ or }  \!B_s\!\in \S_{\text{bot}}, B_t\in \S_{\text{top}}, B[s,t]\subseteq \S\},$$

The times $\eta$ and $\sigma_{\S_{\text{bot}}}$ should typically be close to each other and it would be convenient below to be able to replace $\eta$ by $P^x(S_{T_{D_0}}\in I_k)$. The main difficulty in our estimate of $P^x(S_{T_{D_0}}\in I_k)$ is dealing with the cases where $\max\{\sigma_{\S_{\text{bot}}},\sigma_{\S_{\text{top}}}\}\neq \eta$. This can happen either in the vicinity of $\S_{\text{bot}}$ or if the Brownian path avoids $\S_{\text{bot}}$ but hits $\S_{\text{top}}$ ``from above" and $\S_{\text{bot}}$ from below. 

Note first that
\begin{equation}\label{firstdecomp} 
P^x(S_{T_{D_0}}\in I_k)  \leq  P^x(\sigma_{\mathcal{S}_k^+}\leq \eta) = P^x(\sigma_{\mathcal{S}_k^+}\leq \eta \wedge \sigma_{\S_{\text{end}}})+ P^x( \sigma_{\S_{\text{end}}}\leq \sigma_{\mathcal{S}_k^+}\leq \eta)
\end{equation}
and that 
\begin{equation}\label{simplify}
\begin{aligned}
P^x(\sigma_{\mathcal{S}_k^+}\leq \eta \wedge \sigma_{\S_{\text{end}}})& \leq  P^x(\sigma_{\mathcal{S}_k^+}\leq \sigma_{\S_{\text{bot}}})
+P^x(\sigma_{\L_1}\leq \sigma_{\S_{\text{bot}}}\leq \sigma_{\mathcal{S}_k^+} \leq \eta\wedge \sigma_{\S_{\text{end}}})\\
& \leq  2P^x(\sigma_{\mathcal{S}_k^+}\leq \sigma_{\S_{\text{bot}}}),
\end{aligned}
\end{equation}
where we use the reflection principle at time $\sigma_{\L_1}$ and the fact that for every path starting in $\L_1$ for which $\sigma_{\S_{\text{bot}}}\leq \sigma_{\mathcal{S}_k^+}\leq \eta\wedge \sigma_{\S_{\text{end}}}$ there is a there is a reflection about $\L_1$ for which $\sigma_{\mathcal{S}_k^+}\leq \sigma_{\S_{\text{bot}}}$.

Note that if we define for $2\leq k\leq N-1$, 
$$I_{k, \text{bot}}^{++}=\{z\in\mathcal{S}_{\text{bot}}: z\in \cup_{i=k-2}^{k+2}\mathcal{S}_k\},$$
with the convention $\mathcal{S}_0=\mathcal{S}_{N+1}=\emptyset$, the Markov property at time $\sigma_{\S_{\text{bot}}}$ gives
\begin{equation*}\label{fromStoI}
\begin{aligned}
P^x(\sigma_{\mathcal{S}_k^+}\leq \sigma_{\S_{\text{bot}}})& \leq  P^x(\sigma_{\mathcal{S}_k^+}\leq \sigma_{\S_{\text{bot}}}, B(\sigma_{\S_{\text{bot}}})\in I_{k, \text{bot}}^{++})+  P^x(\sigma_{\mathcal{S}_k^+}\leq \sigma_{\S_{\text{bot}}}, B(\sigma_{\S_{\text{bot}}})\not\in I_{k, \text{bot}}^{++})\\
& \leq  P^x(B(\sigma_{\S_{\text{bot}}})\in I_{k, \text{bot}}^{++}) +P^x(\sigma_{\mathcal{S}_k^+}\leq \sigma_{\S_{\text{bot}}})\bigo{\frac{1}{\log n}},
\end{aligned}
\end{equation*}
so
\begin{equation}\label{fromStoI}
P^x(\sigma_{\mathcal{S}_k^+}\leq \sigma_{\S_{\text{bot}}})\lesssim P^x(B(\sigma_{\S_{\text{bot}}})\in I_{k, \text{bot}}^{++})\leq\frac{(k_0k)^{\ca-1}}{(k^{\ca}-k_0^{\ca})^2\log^{c_\alpha} n},
\end{equation}
with $c_{\alpha}=1/2$, where the second inequality follows from Lemma \ref{beurling1}.

Equations \eqref{simplify} and \eqref{fromStoI} now give a bound for the first term on the right of the equality in \eqref{firstdecomp}: With $c_{\alpha}=1/2$,
\begin{equation}\label{firstterm}
P^x(\sigma_{\mathcal{S}_k^+}\leq \eta \wedge \sigma_{\S_{\text{end}}})\lesssim \frac{(k_0k)^{\ca-1}}{(k^{\ca}-k_0^{\ca})^2\log^{c_\alpha} n},
\end{equation}

 For the second term in that equality, we have
\begin{equation}\label{043020}
P^x( \sigma_{\S_{\text{end}}}\leq \sigma_{\mathcal{S}_k^+}\leq \eta)\leq P^x( \sigma_{\S_{\text{end}}}\leq \sigma_{\mathcal{S}_k^+}\wedge \eta)\sup_{y\in\S_{\text{end}}}P^y(\sigma_{\mathcal{S}_k^+}\leq \eta).
\end{equation}

Note that for 
the last probability in \eqref{043020}, we can't use the argument of \eqref{simplify}. Instead we define, for $1\leq j\leq N$,
$$I_{j,\text{bot}}=\{z\in \S_{\text{bot}}:\re(z+z_0)\in [(j-1)\log^2 n,j\log^2 n)\}.$$
If we define 
\begin{equation}\label{P}
P=\sup_{y\in\S_{\text{end}}}P^y(\sigma_{\mathcal{S}_k^+}\leq \eta),
\end{equation}
then by symmetry, there is a point 
\begin{equation*}\label{y0}
y_0\in\S_{\text{end}}\cap\{z\in D_0:\im(z+z_0)\geq 0\}
\end{equation*}
such that $P=P^{y_0}(\sigma_{\mathcal{S}_k^+}\leq \eta)$. For such a point, there is $0<C<1$ such that 
\begin{equation}\label{travelabit}
P^{y_0}(B(\sigma_{\S_{\text{bot}}})\in I_{1,\text{bot}}, \re(z+z_0)\leq 2b\log n)=C.
\end{equation}

 Applying the strong Markov property at $\sigma_{I_{j,\text{bot}}}$ and $\sigma_{\S_{\text{end}}}$, we have
\begin{equation}\label{043020b}
\begin{aligned}
P & \leq P^{y_0}(\sigma_{\mathcal{S}_k^+}\leq \sigma_{\S_{\text{bot}}})+\sum_{j=1}^N P^{y_0}(B(\sigma_{\S_{\text{bot}}})\in I_{j,\text{bot}}, \sigma_{\S_{\text{bot}}}\leq \sigma_{\mathcal{S}_k^+}\leq \eta)\\
& \leq P^{y_0}(\sigma_{\mathcal{S}_k^+}\leq \sigma_{\S_{\text{bot}}})+\sum_{j=1}^NP^{y_0}(B(\sigma_{\S_{\text{bot}}})\in I_{j,\text{bot}})\sup_{y\in I_{j,\text{bot}}} P^y(\sigma_{\mathcal{S}_k^+}\leq \eta\wedge \sigma_{\S_{\text{end}}})\\
&\hspace{6pc}+\sum_{j=1}^N P^{y_0}(B(\sigma_{\S_{\text{bot}}})\in I_{j,\text{bot}})\sup_{y\in I_{j,\text{bot}}} P^y(\sigma_{\S_{\text{end}}}\leq \sigma_{\mathcal{S}_k^+}\wedge \eta)\cdot P
\end{aligned}
\end{equation}

Note that Lemma \ref{beurling1} implies that for $j\geq 2$,
\begin{equation}\label{01-29-1}
\sup_{y\in\S_{\text{end}}}P^y(B(\sigma_{\S_{\text{bot}}})\in I_{j,\text{bot}})\lesssim  \frac{1}{j^{3/2}\log n}
\end{equation}
and 
\begin{equation}\label{05-01-20}
\sup_{y\in I_{j,\text{bot}}} P^y(\sigma_{\mathcal{S}_k^+}\leq \eta\wedge \sigma_{\S_{\text{end}}}) \lesssim \frac{(jk)^{-1/2}}{(k^{1/2}-j^{1/2})^2\log^{1/2} n}.
\end{equation}
Using \eqref{travelabit}, \eqref{01-29-1}, and \eqref{05-01-20} in \eqref{043020b}, we see that
\begin{eqnarray*}
P &\lesssim & P^{y_0}(\sigma_{\mathcal{S}_k^+}\leq \sigma_{\S_{\text{bot}}})\\
&&\hspace{2pc} +\sup_{y\in I_{1,\text{bot}}} P^y(\sigma_{\mathcal{S}_k^+}\leq \eta\wedge \sigma_{\S_{\text{end}}})+\sum_{j=2}^N\frac{1}{j^{3/2}\log n}\frac{(jk)^{-1/2}}{(k^{1/2}-j^{1/2})^2\log^{1/2} n}\\
&&\hspace{4pc}+CP+\sup_{\substack{y\in I_{1,\text{bot}},\\\re(y+z_0)\geq 2b\log n}} P^y(\sigma_{\S_{\text{end}}}\leq \sigma_{\mathcal{S}_k^+}\wedge \eta)P\\
&&\hspace{6pc}+\sum_{j=2}^N\frac{1}{j^{3/2}\log n}\sup_{y\in I_{j,\text{bot}}} P^y(\sigma_{\S_{\text{end}}}\leq \sigma_{\mathcal{S}_k^+}\wedge \eta)\cdot P, 
\end{eqnarray*}
which implies, since $\sum_{j=2}^N\frac{1}{j^{3/2}\log n}=o(1)$ and $\displaystyle{\sup_{\substack{y\in I_{1,\text{bot}},\\\re(y+z_0)\geq 2b\log n}} P^y(\sigma_{\S_{\text{end}}}\leq \sigma_{\mathcal{S}_k^+}\wedge \eta)}<1$, that 
\begin{eqnarray*}
P &\lesssim & P^{y_0}(\sigma_{\mathcal{S}_k^+}\leq \sigma_{\S_{\text{bot}}})+\sup_{y\in I_{1,\text{bot}}} P^y(\sigma_{\mathcal{S}_k^+}\leq \eta\wedge \sigma_{\S_{\text{end}}})+\sum_{j=2}^N\frac{1}{j^{3/2}\log n}\frac{(jk)^{-1/2}}{(k^{1/2}-j^{1/2})^2\log^{1/2} n}\\
\end{eqnarray*}
Using the same argument as in \eqref{fromStoI} and the fact that $\sum_{j=2}^N\frac{1}{j^{2}}\frac{1}{(k^{1/2}-j^{1/2})^2}\lesssim k^{-1}$ (which can be shown, for instance, with help of the ideas in the second half of the proof of Proposition \ref{expdiff} below), we get 
\begin{equation}\label{Pbound}
P \lesssim  \frac{k^{-3/2}}{\log^{1/2} n}+\frac{1}{\log^{3/2} n}\sum_{j=2}^N\frac{1}{j^{2}}\frac{1}{(k^{1/2}-j^{1/2})^2}\lesssim  \frac{k^{-3/2}}{\log^{1/2} n}.
\end{equation}
Some of the ideas used in the argument leading to the bound in \eqref{Pbound} also imply
\begin{equation}\label{050320}
P^x( \sigma_{\S_{\text{end}}}\leq \sigma_{\mathcal{S}_k^+}\wedge \eta)\lesssim  \frac{k_0^{-3/2}}{\log^{1/2} n}.
\end{equation}
It now follows from \eqref{043020}, \eqref{P}, \eqref{Pbound}, and \eqref{050320} that
\begin{equation}\label{final}
P^x( \sigma_{\S_{\text{end}}}\leq \sigma_{\mathcal{S}_k^+}\leq \eta)\lesssim \frac{(k_0k)^{-3/2}}{\log  n}. 
\end{equation}
Since this bound is smaller than that in \eqref{firstterm}, combining \eqref{firstdecomp}, \eqref{firstterm}, and \eqref{final} yields the lemma when $\alpha=0$ and $2\leq k\leq N-1$. A similar argument can be used to handle the cases $k=1$ and $k=N$, which completes the proof in the case $\alpha = 0$ and thus of the lemma.
\end{proof}

\subsection{Obtaining the Rate of Convergence for the Green's Function}

The following consequence of the lemmas of the previous section is a key element of the proof of Theorem \ref{green}:

\begin{prop}\label{expdiff}
Suppose $x\in D_{\alpha}\cap\Z^2,y\in D_{\alpha}$ are such that $d(x,\bd \Da)\leq 4c_0\log n$, $d(x,y)\leq 3c_0\log n$, and a subset of $\bd \Da$ closest to $x$ is $I_{k_0}$. Then 
 
 $$E^{x,y}\left[\bigg|\log \frac{|S_{T_{D_{\alpha}}}|}{|B_{\tau_{D_{\alpha}}}|}\bigg|\right]\lesssim k_0^{c_\alpha-1}n^{-c_{\alpha}}\log^{c_{\alpha}+1} n,$$
where 
$$c_{\alpha}=\frac{1}{2}+\frac{\alpha}{4\pi-2\alpha}.$$
\end{prop}

\begin{proof} If for $1\leq k, \ell \leq N, z\in I_k, w\in I_\ell$, then 
$$|\log(|z|/|w|)|\leq \left|\log\left(\frac{|w|+|z-w|}{|w|}\right)\right|=\bigo{\frac{(|k-\ell|+1)\log^2 n}{n}},$$
so
\begin{eqnarray}\label{last1}\notag
E^{x,y}\left[\bigg|\log \frac{|S_{T_{D_{\alpha}}}|}{|B_{\tau_{D_{\alpha}}}|}\bigg|\right] 
& \lesssim & \frac{\log^2 n}{n}\sum_{k,\ell=1}^N P^x(S_{T_{D_{\alpha}}}\in I_k)P^y(B_{\tau_{D_{\alpha}}}\in I_\ell)(|k-k_0|+|\ell-k_0|+1)\\ \notag
& \lesssim & \frac{\log^2 n}{n}\left(1+\sum_{k=1}^N P^x(S_{T_{D_{\alpha}}}\in I_k)|k-k_0|+\sum_{\ell=1}^N P^y(B_{\tau_{D_{\alpha}}}\in I_\ell)|\ell-k_0|\right)\\
&\leq &\frac{\log^2 n}{n}+\sum_{k=1}^{N}\frac{(k_0 k)^{c_\alpha-1}}{(k_0^{c_\alpha}-k^{c_\alpha})^2\log^{c_\alpha} n}\frac{|k_0-k|\log^2 n}{n},
\end{eqnarray}
where the last inequality is a consequence of Lemmas \ref{beurling1} and  \ref{beurling2}.
If $N\geq \lceil 3k_0/2\rceil$, then
\begin{equation}\label{last2}
\sum_{\substack{k =1\\ k\neq k_0}}^N\frac{k^{c_\alpha - 1}|k_0-k|}{(k_0^{c_\alpha}-k^{c_\alpha})^2}=S_1+S_2+S_3+S_4,
\end{equation}
where
$$S_1=\sum_{k\leq \lfloor k_0/2\rfloor}\frac{k^{c_\alpha - 1}(k_0-k)}{(k_0^{c_\alpha}-k^{c_\alpha})^2}, \qquad S_2=\sum_{k=\lceil k_0/2\rceil}^{k_0-1}\frac{k^{c_\alpha - 1}(k_0-k)}{(k_0^{c_\alpha}-k^{c_\alpha})^2},$$
$$S_3=\sum_{k=k_0+1}^{\lfloor 3k_0/2\rfloor}\frac{k^{c_\alpha - 1}(k-k_0)}{(k^{c_\alpha}-k_0^{c_\alpha})^2}, \qquad S_4=\sum_{k=\lceil 3k_0/2\rceil}^{N}\frac{k^{c_\alpha - 1}(k-k_0)}{(k^{c_\alpha}-k_0^{c_\alpha})^2}.$$
Clearly, $S_1\lesssim k_0^{1-c_{\alpha}}$. It is easy to see that for $\lceil k_0/2\rceil \leq k \leq k_0-1$, 
$$\frac{k_0-k}{(k_0^{c_\alpha}-k^{c_\alpha})^2}\lesssim \frac{k_0^{2-2c_\alpha}}{k_0-k},$$
so 
$$S_2\lesssim \sum_{j=1}^{\lfloor k_0/2\rfloor}\frac{(k_0-j)^{c_{\alpha}-1}k_0^{2-2c_\alpha}}{j}\lesssim k_0^{1-c_\alpha}\log (k_0).$$
Similarly, $S_3\lesssim k_0^{1-c_\alpha}\log (k_0)$. Finally, using the fact that if $k\geq \lceil 3k_0/2\rceil$, there is $C>0$ such that $(1-(k_0/k)^{c_\alpha})\geq C$, we see that 
$$S_4\lesssim \sum_{k=\lceil 3k_0/2\rceil}^{N}k^{-c_\alpha}\lesssim N^{1-c_\alpha}-(3k_0/2)^{1-c_\alpha}.$$
Combining the bounds for $S_1, S_2, S_3, S_4$ with \eqref{last2} and \eqref{last1}, we get, in the case $N\geq \lceil 3k_0/2\rceil$,
$$\sum_{\substack{k =1\\ k\neq k_0}}^N\frac{k^{c_\alpha - 1}|k_0-k|}{(k_0^{c_\alpha}-k^{c_\alpha})^2}\lesssim N^{1-c_\alpha} \log N,$$
so 
$$E^{x,y}\left[\bigg|\log \frac{|S_{T_{D_{\alpha}}}|}{|B_{\tau_{D_{\alpha}}}|}\bigg|\right]\lesssim \frac{k_0^{c_\alpha -1}\log^{2-c_\alpha} n}{n}N^{1-c_\alpha} \log N\lesssim \frac{k_0^{c_\alpha -1}\log{2-c_\alpha} n}{n}\left(\frac{n}{\log^2 n}\right)^{1-c_\alpha} \log n,
$$
which proves the proposition in the case where $N\geq \lceil 3k_0/2\rceil$. If $N< \lceil 3k_0/2\rceil$, $S_4=0$ and the bounds for $S_1, S_2, S_3$ are the same.
\end{proof}





The strategy we use in the proof of Theorem \ref{green} is similar to that in \cite{kl} and \cite{bjk}, though we handle some technical issues slightly differently in the present paper. We couple $B$ and $S$ until they are close to $\bd\Da$ but are likely not to have left $\Da$ yet. We then let each run independently and use Proposition \ref{expdiff}. The technical difficulty stems in the fact that we would like to use the strong Markov property but $B$ and $S$ are not strong Markov when considered jointly under the coupling.

\begin{proof} [Proof of Theorem \ref{green}]

In light of \eqref{discretegreenrep}, \eqref{ax}, and \eqref{contgreenrep} above,  since there exist $c_1, c_2>0$ such that 
\begin{equation}\label{inrad}
c_1n\leq \inf_{z \in D_{\alpha}}|z|\leq \sup_{z \in D_{\alpha}}|z|\leq c_2n,
\end{equation}
we have
\begin{equation}\label{mainproof1}
\left|G_{D_{\alpha}}(w)-\frac{2}{\pi}g_{D_{\alpha}}(w)\right|=E^w\left[\bigg|\log \frac{|S_{T_{D_{\alpha}}}|}{|B_{\tau_{D_{\alpha}}}|}\bigg|\right]+\bigo{|w|^{-2}}.
\end{equation}
In order to prove the theorem, we just need to show that the upper bound of Proposition \ref{expdiff} with $k_0^{c_{\alpha}-1}$ replaced by the obvious upper bound of 1 also holds for the expected value in \eqref{mainproof1}.  We let $\mathcal{K}$ be as in \eqref{k}. Note that \eqref{inrad} implies that 
$$\bigg|\log \frac{|S_{T_{D_{\alpha}}}|}{|B_{\tau_{D_{\alpha}}}|}\bigg|\asymp 1,$$
so
$$E^w\left[\bigg|\log \frac{|S_{T_{D_{\alpha}}}|}{|B_{\tau_{D_{\alpha}}}|}\bigg|\right] \leq E^w\left[\bigg|\log \frac{|S_{T_{D_{\alpha}}}|}{|B_{\tau_{D_{\alpha}}}|}\bigg|; \K\right]+ \bigo{n^{-10}}$$

For the rest of the proof, we let $c_0$ be as in \eqref{k} 
and define 
$$\eta= \inf\{t\geq 0: d(B_t,\partial \Da)\leq 2c_0\log n\}$$
(see Figure \ref{Fig1}) and note that on the event $\K$, $\eta\leq \tau_{D_\alpha}\wedge T_{D_\alpha}$.
Then 
\begin{eqnarray*}
E^w\left[\bigg|\log \frac{|S_{T_{D_{\alpha}}}|}{|B_{\tau_{D_{\alpha}}}|}\bigg|; \K\right] &=& \sum_{k\geq 0} E^w\left[\bigg|\log \frac{|S_{T_{D_{\alpha}}}|}{|B_{\tau_{D_{\alpha}}}|}\bigg|; \K; \eta\in [k,k+1)\right]\\
& \leq & \sum_{k\geq 0} E^w\left[\bigg|\log \frac{|S_{T_{D_{\alpha}}}|}{|B_{\tau_{D_{\alpha}}}|}\bigg|; \mathcal{E}\right],
\end{eqnarray*}
where $\mathcal{E}=\{\max\{d(S_k, \bd \Da), d(B_k, \bd \Da)\}\leq 4c_0\log n; d(B_k,S_k)\leq 3c_0\log n; k\leq \tau_{D_\alpha}\wedge T_{D_\alpha}\}$ and the last inequality follows from Lemma \ref{toofar}. The proof of the theorem is now complete if we apply the Markov property at time $k$ and use Proposition \ref{expdiff}.


\end{proof}

\section{Acknowledgments}
The author gratefully acknowledges support through PSC-CUNY Award \# 61514-00 49.

\bibliographystyle{plain}
\bibliography{harmonic}

\end{document}

%% file: defns.tex
\newcommand{\C}{\mathbbm{C}}   
\newcommand{\R}{\mathbbm{R}}                 
 
\newcommand{\Z}{\mathbbm{Z}}                 

\renewcommand{\L}{\mathcal{L}}

\newcommand{\1}{\mathbbm{1}}
\newcommand{\D}{\mathbb{D}}

\newcommand{\im}{\text{Im}}
\newcommand{\re}{\text{Re}}

\newcommand{\eps}{\epsilon}

\newcommand{\om}{\omega}

\newcommand{\bigo}[1]{\mathcal{O}\left(#1\right)}

\renewcommand{\S}{\mathcal{S}}

\newcommand{\bd}[1]{\partial{#1}}

\newcommand{\h}{\mathbbm{H}}

\newcommand{\Da}{D_{\alpha}}
\newcommand{\Dap}{D'_{\alpha}}
\newcommand{\ca}{c_{\alpha}}
\newcommand{\K}{\mathcal{K}}

